\documentclass[12pt,a4paper]{article}
\usepackage{amsmath,amsthm,amssymb,mathrsfs,enumerate}
\usepackage[dvipdfmx]{hyperref}
\usepackage[numbers]{natbib}

\def\A{{\mathcal{A}}}
\def\B{{\mathcal{B}}}

\def\N{{\mathbb{N}}}
\def\R{{\mathbb{R}}}

\theoremstyle{plain}
\newtheorem{thm}{Theorem}[section]
\newtheorem{lem}{Lemma}[section]

\newtheorem{prop}{Proposition}[section]

\theoremstyle{definition}

\newtheorem{rem}{Remark}[section]
\newtheorem{assmp}{Assumption}[section]
\newtheorem*{assmp*}{Assumption 3.4$'$}

\numberwithin{equation}{section}
\allowdisplaybreaks

\title{Recursive Variational Problems in Nonreflexive Banach Spaces with an Infinite Horizon: An Existence Result\thanks{I am grateful to Dean Carlson and H\'{e}l\`{e}ne Frankowska for helpful comments to an earlier version of the manuscript. This research is supported by JSPS KAKENHI Grant No.\ 26380246 from the Ministry of Education, Culture, Sports, Science and Technology, Japan.}}
\date{\today}
\author{Nobusumi Sagara \\[-3pt]
\\[-3pt]
{\small Faculty of Economics, Hosei University} \\[-4pt]
{\small 4342, Aihara, Machida, Tokyo, 194-0298, Japan} \\[-4pt]
{\footnotesize e-mail: nsagara@hosei.ac.jp}}

\begin{document}
\maketitle
\setcounter{page}{0}
\thispagestyle{empty}
\begin{abstract} 
We investigate variational problems with recursive integral functionals governed by infinite-dimensional differential inclusions with an infinite horizon and present an existence result in the setting of nonreflexive Banach spaces. We find an optimal solution in a Sobolev space taking values in a Banach space under the Cesari type condition. We also investigate sufficient conditions for the existence of solutions to the initial value problem for the differential inclusion. \\


\end{abstract}
\clearpage

\section{Introduction}
Recursive integral functionals with an infinite horizon introduced in \cite{uz68} and then elaborated by \cite{ep87a,ep87b,eh83} endogenize intertemporally variable discount factors, which is a continuous time analogue of \cite{ko60} in the discrete time formulation of the specific form
$$
\sum_{t=0}^\infty L(x_t,x_{t+1})\exp\left( \sum_{s=0}^tf(x_s,x_{s+1}) \right)
$$ 
where $L$ is a cost/utility function and $f$ is a discount function. Contrary to continuous/discrete time models with constant discount rates, those with recursive objective functionals exhibit intriguing dynamics in economic growth models ranging from the saddle point stability with a unique stationary state (see \cite{chl91,ch94,cns08,ep87a,ep87b,eh83,na84,ob90,se93,uz68}), the existence of balanced growth paths (see \cite{pwz97}), and to possible complex dynamics with multiple stationary states (see \cite{da03,dr96,dw07,els11,iw72,rh73}), depending upon the assumption on the discount function. Needless to say, dynamic economic analysis with continuous time in the above literature hinges upon the Pontraygin's maximum principle and the Hamilton--Jacobi--Bellman equations; see \cite{bb92,bhm90,chl91,so91} for necessary conditions for optimality in a more general framework for the case with recursive integral functionals. 

It is thus quite natural to pursue the existence of optimal solutions to variational and optimal control problems with recursive integral functionals as generally as possible to guarantee the dynamical analysis to be meaningful enough. There are several attempts addressed to this issue. The first existence result with recursive integral functionals appeared in \cite{bbs89} as a convex variational problem and then whose result was translated into optimal control problems with some improvement in \cite{ba93,ca90} under the Cesari type of continuity and convexity assumptions. Existence results without convexity assumptions were established in \cite{sa01,sa07}, which detours the relaxation technique by virtue of employing the norm topology of weighted Sobolev spaces on which certain boundedness assumptions are imposed to ensure the norm compactness of the set of admissible arcs. All of these works were, however, devoted to the finite-dimensional control systems. 

In this paper, we investigate variational problems with recursive integral functionals governed by infinite-dimensional differential inclusions with an infinite horizon and present an existence result in the setting of nonreflexive Banach spaces under the Cesari type condition. An intricate difficulty arises in the framework under consideration for the compactness of the set of admissible arcs, which stems exclusively from nonreflexivity and unbounded interval for infinite-dimensional control systems. In the infinite horizon setting, existence results for optimal control problems governed by linear evolution equations taking values in a Hilbert space $H$ were explored in \cite{bp12,chj87,za00}. The controls in these works are $L^2$-functions with values in $H$, which significantly simplifies the compactness argument because bounded sets in $L^2$ are relatively weakly compact; thereby the standard diagonalization procedure for a minimizing sequence works smoothly to obtain an optimal arc as its limit. We find an optimal arc in a Sobolev space taking values in a Banach space $E$ such that the derivatives are $L^1$-functions with values in $E$. To dispense with reflexivity for the compactness argument, we employ Diestel's theorem (see \cite{drs93}) on the weak compactness in the space of Bochner integrable functions. Then the uniform convergence of a minimizing sequence follows from the Arzela--Ascoli theorem.  

The organization of the paper is as follows. After providing mathematical preliminaries in Section 2, we present an existence result in Section 3. We also investigate in Section 4 sufficient conditions for the existence of solutions to the initial value problem governed by the differential inclusion.

\section{Preliminaries}
\subsection{Sobolev Spaces}
Let $(E,\| \cdot \|)$ be a Banach space and $E^*$ be the dual space of $E$ with the duality denoted by $\langle x^*,y \rangle$ for $x^*\in E^*$ and $y\in E$. Denote by $\Omega=[0,\infty)$ an unbounded interval of the real line with the Lebesgue measure. A function $x:\Omega\to E$ is said to be \textit{locally absolutely continuous} if its restriction to the bounded closed interval $[0,T]$ is absolutely continuous for every $T>0$, i.e., for every $T>0$ and $\varepsilon>0$ there exists $\delta>0$ such that $0\le t_0<t_1<\dots <t_n\le T$ and $\sum_{i=1}^n| t_{i-1}-t_i|<\delta$ imply $\sum_{i=1}^n\| x(t_{i-1})-x(t_i) \|<\varepsilon$. A function $x$ is said to be \textit{differentiable} at $t_0>0$ if there exists $\xi\in E$ such that 
$$
\lim_{t\to t_0} \left\| \frac{x(t)-x(t_0)}{t-t_0}-\xi \right\|=0. 
$$
The vector $\xi$ is denoted by $x'(t_0)$ and called the \textit{derivative} of $x$ at $t_0$. It should be noted that unlike the real-valued case, locally absolutely continuous functions with values in Banach spaces fail to be differentiable almost everywhere; see \cite[Examples 1 and 2]{pu88} or \cite[Example 4.2]{de92} for such examples. The nondifferentiability of locally absolutely continuous functions disappears under the reflexivity assumption. Specifically, every locally absolutely function $x:\Omega\to E$ has the derivative $x'(t)$ a.e.\ $t\in \Omega\setminus \{ 0 \}$ with $x(t)=\int_0^tx'(s)ds+x(0)$ for every $t\in \Omega$ whenever $E$ is reflexive; see \cite[Lemma]{ko67}. We do not assume, however, the reflexivity of $E$ throughout the paper. 

A function $x:\Omega\to E$ is said to be \textit{strongly measurable} if there exists a sequence of simple functions $\{ x_n \}$ from $\Omega$ to $E$ such that $\| x_n(t)-x(t) \|\to 0$ a.e.\ $t\in \Omega$. A strongly measurable function $x$ is \textit{locally Bochner integrable} if it is Bochner integrable on every compact subset of $A\subset \Omega$, that is, $\int_A\| x(t) \|dt<\infty$, where the \textit{Bochner integral} of $x$ over $A$ is defined by $\int_Ax(t)dt:=\lim_n\int_Ax_n(t)dt$. Let $L^1_{\mathrm{loc}}(\Omega,E)$ be the space of (the equivalence classes of) locally Bochner integrable functions from $\Omega$ to $E$. Denote by $W^{1,1}_{\mathrm{loc}}(\Omega,E)$ the Sobolev space, which consists of all locally absolutely continuous functions $x:\Omega\to E$ whose derivative $x'(t)$ exists a.e.\ $t\in \Omega\setminus \{ 0 \}$ such that there exists $u\in L^1_{\mathrm{loc}}(\Omega,E)$ satisfying $x(t)=\int_0^tu(s)ds+x(0)$ for every $t\in \Omega$ with $u(t)=x'(t)$ a.e.\ $t\in \Omega\setminus \{ 0 \}$. For each $n\in \N$, define the seminorm $\mu_n$ on $W^{1,1}_{\mathrm{loc}}(\Omega,E)$ by $\mu_n(x)=\int_0^n(\| x(t) \|+\| x'(t) \|)dt$. Since $\{ \mu_n \}_{n\in \N}$ is a countable separating family of seminorms, $W^{1,1}_{\mathrm{loc}}(\Omega,E)$ is a Fr\'echet space under the compatible metric $d$ given by
$$
d(x_1,x_2)=\max_{n\in \N}\frac{\mu_n(x_1-x_2)}{2^n(1+\mu_n(x_1-x_2))}, \quad x_1,x_2\in W^{1,1}_{\mathrm{loc}}(\Omega,E).  
$$
When $\Omega$ is replaced by $[0,T]$, the above definition simply leads to that of the Sobolev space $W^{1,1}([0,T],E)$ normed by $\| x \|_{1,1}=\int_0^T(\| x(t) \|+\| x'(t) \|)dt$.

\subsection{Weak Compactness in $L^1$}
Let $L^1([0,T],E)$ be the space of (the equivalence classes of) Bochner integrable functions from $[0,T]$ to $E$ normed by $\| x\|_1=\int_0^T\| x(t) \|dt$. A function $p:[0,T]\to E^*$ is said to be \textit{weakly$^*\!$ scalarly measurable} if $t\mapsto \langle p(t),y \rangle$ is measurable for every $y\in E$. Weakly$^*\!$ scalarly measurable functions $p$ and $q$ are said to be \textit{weakly$^*\!$ scalarly equivalent} if $\langle p(t),y \rangle=\langle q(t),y \rangle$ for every $y\in E$ a.e.\ $t\in [0,T]$, where the null set on which the equality fails depends upon $y$. Denote by $L^\infty_{w^*}([0,T],E^*)$ the space of (the weak$^*\!$ scalar equivalence classes of) weakly$^*\!$ scalarly measurable functions $p:[0,T]\to E^*$ such that there exists $c\ge 0$ with $| \langle p(t),y \rangle|\le c\| y \|$ a.e.\ $[0,T]$ for every $y\in E$, where the null set on which the inequality fails depends upon $y$. The infimum of all constants $c$ such that the inequality holds for every $y\in E$ is denoted by $\| p \|_\infty$, for which we have $\| p(t) \|\le \| p \|_\infty$ a.e.\ $t\in [0,T]$. The dual space of $L^1([0,T],E)$ is $L^\infty_{w^*}([0,T],E^*)$ with the norm $\| \cdot \|_\infty$ under the duality $\langle p,x\rangle=\int_0^T\langle p(t),x(t) \rangle dt$ for $p\in L^\infty_{w^*}([0,T],E^*)$ and $x\in L^1([0,T],E)$; see \cite[Theorem 12.2.11]{fa99}. The norm $\| p \|_\infty$ on $L^\infty_{w^*}([0,T],E^*)$ may not coincide with $\mathrm{ess.\,sup}_{t\in [0,T]}\| p(t) \|$, but such a discrepancy disappears whenever $E$ is separable; see \cite[Example 12.2.1 and Lemma 12.2]{fa99}. A subset $K$ of $L^1([0,T],E)$ is said to be \textit{uniformly integrable} if 
$$
\lim_{|A|\to 0}\sup_{x\in K}\int_A\| x(t) \|dt=0
$$
where $|A|$ is the Lebesgue measure of $A\subset [0,T]$. A set-valued mapping with nonempty values is called a \textit{multifunction}. A multifunction $G:[0,T]\twoheadrightarrow E$ is said to be integrably bounded if there exists $\gamma\in L^1([0,T])$ such that $\| x \|\le \gamma(t)$ for every $x\in G(t)$ and $t\in [0,T]$. 

The following weak compactness result is fundamental to the analysis in the sequel, which requires neither the reflexivity nor the separability of $E$. 

\begin{thm}[\citet{drs93}]
\label{diest}
If $K$ is a bounded, uniformly integrable subset of $L^1([0,T],E)$ such that there is a relatively weakly compact-\hspace{0pt}valued multifunction $G:[0,T]\twoheadrightarrow E$ with $x(t)\in G(t)$ for every $x\in K$ and $t\in [0,T]$, then $K$ is relatively weakly compact in $L^1([0,T],E)$. 
\end{thm} 

\noindent
It is easy to see that a sufficient condition for the boundedness and uniform integrability of $K$ is \textit{integrable boundedness}: There exists $\varphi\in L^1([0,T])$ such that $\| x(t) \|\le \varphi(t)$ for every $x\in K$ and $t\in [0,T]$. 

The next result is useful for later use.

\begin{lem}
\label{lem1}
If $\{ x_n \}_{n\in \N}$ is a sequence in $L^1([0,T],E)$ converging weakly to $x$, then for each $n\in \N$ there exists $\hat{x}_n\in L^1([0,T],E)$ such that $\hat{x}_n$ is a convex combination of $\{ x_k\mid k\ge n \}$ and $\hat{x}_{n(k)}(t)\to x(t)$ strongly in $E$ a.e.\ $t\in [0,T]$ for some subsequence $\{ \hat{x}_{n(k)} \}_{k\in \N}$ of $\{ \hat{x}_n \}_{n\in \N}$. 
\end{lem}

\begin{proof}
By Mazur's lemma (see \cite[Corollary V.3.14]{ds58}), for each $n\in \N$ there exist a sequence $\{ x^n_i \}_{i\in \N}$ in $L^1([0,T],E)$ such that $x^n_i$ is a convex combination of $\{ x_k\mid k\ge n \}$ with $x^n_i\to x$ strongly in $L^1([0,T_1],E)$. Hence, one can find $i(n)\in \N$ such that $\| x^n_{i(n)}-x \|_1<1/n$. Let $\hat{x}_n=x^n_{i(n)}$. Then by construction, $\hat{x}_n$ is a convex combination of $\{ x_k\mid k\ge n \}$ with $\hat{x}_n\to x$ strongly in $L^1([0,T_1],E)$. Therefore, one can extract a subsequence $\{ \hat{x}_{n(k)} \}_{k\in \N}$ of $\{ \hat{x}_n \}_{n\in \N}$ satisfying $\hat{x}_{n(k)}(t)\to x(t)$ strongly in $E$ a.e.\ $t\in [0,T]$. 
\end{proof}

\subsection{The Space of Continuous Functions}
If $X$ is a weakly compact subset of a separable Banach space $E$, then $X$ is metrizable with respect to the weak topology of $E$; see \cite[Theorem V.6.3]{ds58}. In this case, denote by $C_w(\Omega,X)$ (resp.\ $C(\Omega,X)$) the space of continuous functions from $\Omega$ to $X$ with respect to the weak (resp.\ norm) topology, endowed with the topology of uniform convergence on compacta. Since $X$ is a compact metric space for the weak topology,  $C_w(\Omega,X)$ is metrized via
$$
d(x_1,x_2)=\sup_{t\in \Omega}\rho(x_1(t),\rho_2(t)), \quad x_1,x_2\in C_w(\Omega,X)
$$
where $\rho$ is a consistent metric on $X$ for the weak topology; see \cite[Definition 8.1]{du66}. By definition, we also have $C(\Omega,X)\subset C_w(\Omega,X)$. Let $B^\varepsilon_w(y)$ (resp.\ $B^\varepsilon(y)$) be the $\varepsilon$-neighborhood of $y\in X$ with respect to the weak (resp.\ norm) topology. A subset $K$ of $C_w(\Omega,X)$ (resp.\ $C(\Omega,X))$ is said to be \textit{equicontinuous} at $t\in \Omega$ with respect to the weak (resp.\ norm) topology if for every $\varepsilon>0$ there exists a $\delta$-neighborhood $B^\delta(t)$ of $t$ such that $x(B^\delta(t))\subset B^\varepsilon_w(x(t))$ (resp.\ $x(B^\delta(t))\subset B^\varepsilon(x(t))$) for every $x\in K$. A subset $K$ of $C_w(\Omega,X)$ (resp.\ $C(\Omega,X))$ is said to be equicontinuous on $\Omega$ with respect to the weak (resp.\ norm) topology if it is equicontinuous with respect to the weak (resp.\ norm) topology at each point in $\Omega$. It is easy to see that if $K\subset C(\Omega,X)$ is equicontinuous with respect to the norm topology, then so it is with respect to the weak topology. Hence, if $K\subset C(\Omega,X)$ is a family of locally absolutely continuous functions, then it is equicontinuous both with respect to the norm and the weak topologies whenever $X$ is a weakly compact subset of a separable Banach space $E$.

\section{Existence of Optimal Arcs}
\subsection{Assumptions}


The recursive variational problem under investigation is:
\begin{equation}
\label{P}
\begin{aligned}
& \min_{x\in W^{1,1}_{\mathrm{loc}}(\Omega,E)}\int_\Omega L(t,x(t),x'(t))F\left( t,\int_0^tf(s,x(s),x'(s))ds \right)dt \\
& \hspace{0.5cm} x'(t)\in \Gamma(t,x(t)) \text{ a.e.\ $t\in \Omega$ and }x(0)=\xi_0\in X
\end{aligned}
 \tag{P}
\end{equation}
where $L:\Omega\times E\times E\to (-\infty,+\infty]$ is a cost function, $f:\Omega\times E\times E\to (-\infty,+\infty]$ is a discount function, $F:\Omega\times \R\to \R$ corresponds to a generalized form of the exponential function, and $\Gamma:\Omega\times X\twoheadrightarrow E$ describes a constraint governed by the multivalued dynamical system. The graph of $\Gamma$ is denoted by $\mathrm{gph}\,\Gamma=\{ (t,y,z)\in \Omega\times X\times E\mid z\in \Gamma(t,y) \}$. The set of \textit{admissible arcs} is given by
$$
\A=\{ x\in W^{1,1}_{\mathrm{loc}}(\Omega,E) \mid x'(t)\in \Gamma(t,x(t)) \text{ a.e.\ $t\in \Omega$ and }x(0)=\xi_0\}. 
$$
An admissible arc that is a solution to \eqref{P} is called an \textit{optimal arc}.

Throughout the paper, $E$ is assumed to be a separable Banach space. Assumptions for the primitive $\{ L,F,f,X,\Gamma \}$ are given below.  

\begin{assmp}
\label{assmp1}
$X$ is a weakly compact subset of $E$.
\end{assmp}

\begin{assmp}
\label{assmp2}
\begin{enumerate}[\rm (i)] 
\item There is an integrably bounded, relatively weakly compact-valued multifunction $G:\Omega\twoheadrightarrow E$ such that $\Gamma(t,y)\subset G(t)$ for every $(t,y)\in \Omega\times X$. 
\item $L(t,y,z)F(t,\cdot)$ is a nondecreasing function on $\R$ for every $(t,y,z)\in \mathrm{gph}\,\Gamma$.
\item $L$ and $F$ are measurable functions such that there exist $\alpha_1,\alpha_2\in L^1(\Omega)$ and $a\in \R$ satisfying 
$$
|L(t,y,z)F(t,r)|\le \alpha_1(t)+\alpha_2(t)\| y \|+a \| z \|
$$
for every $(t,y,z)\in \mathrm{gph}\,\Gamma$ and $r\in \R$. 
\item $f$ is a measurable function such that there exist $\beta_1,\beta_2\in L^1_{\mathrm{loc}}(\Omega)$ and $b\in \R$ satisfying 
$$
|f(t,y,z)|\le \beta_1(t)+\beta_2(t)\| y \|+b \| z \|
$$
for every $(t,y,z)\in \mathrm{gph}\,\Gamma$. 
\end{enumerate}
\end{assmp}

Since there exists $\gamma\in L^1(\Omega)$ such that $\| x'(t) \|\le \gamma(t)$ for every $x\in \A$, we have $\| x(t) \|\le \| \xi_0\|+\int_0^\infty \gamma(s)ds:=c<+\infty$. Hence, letting $\alpha(t):=\alpha_1(t)+c\alpha_2(t)+a\gamma(t)$ and $\beta(t):=\beta_1(t)+c\beta_2(t)+b\gamma(t)$ yields 
\begin{equation}
\label{eq1}
\left| L(t,x(t),x'(t))F\left( t,\int_0^tf(s,x(s),x'(s))ds \right) \right|\le  \alpha(t) \quad\text{a.e.\ $t\in \Omega$} 
\end{equation}
for every $x\in \A$ with $\alpha\in L^1(\Omega)$ and 
\begin{equation}
\label{eq2}
| f(t,x(t),x'(t))|\le \beta(t) \quad\text{a.e.\ $t\in \Omega$} 
\end{equation}
for every $x\in \A$ with $\beta\in L^1_{\mathrm{loc}}(\Omega)$. This guarantees that the recursive integral functional for \eqref{P} is bounded on $\A$ under Assumption \ref{assmp2} whenever $\A$ is nonempty.

\begin{assmp}
\label{assmp3}
$\A$ is nonempty.
\end{assmp}
\noindent
In many economic applications, it is reasonable to impose $0\in \Gamma(t,\xi_0)$ for every $t\in \Omega$ (the possibility of ``inactivity''), which implies that the constant arc $x(t)\equiv \xi_0$ is feasible, or there exists $r>0$ such that $-r e^{-rt}\xi_0\in \Gamma(t,e^{-rt}\xi_0)$ for every $t\in \Omega$, which implies that the exponential decay $x(t)=e^{-rt}\xi_0$ is feasible. We do not assume here such simple conditions. Instead, we provide in Section 4 a sufficient condition for the existence of admissible arcs.   

Define the multifunction $\tilde{\Gamma}:\Omega\times \R\times X\twoheadrightarrow \R\times \R\times E$ by
$$
\tilde{\Gamma}(t,r,y)=\left\{ (a^1,a^2,z)\in \R\times \R\times E \left| \begin{array}{l} a^1\ge L(t,y,z)F(t,r) \\ a^2\ge f(t,y,z),\,z\in \Gamma(t,y) \end{array}
 \right.\right\}.
$$
Then $\tilde{\Gamma}(t,r,y)$ is the augmented velocity set. Denote by $\overline{\mathrm{co}}\,\{ \cdots \}$ the closed convex hull in $\R\times \R\times E$ and let $B^\delta_\mathit{w}(r,y)$ be the open ball with center $(r,y)\in \R\times X$ with radius $\delta>0$ with respect to the metric $\rho$ on $X$ consistent with the weak topology of $X$.  

\begin{assmp}[Cesari property]
\label{ces}
For every $(t,r,y)\in \Omega\times \R\times X$: 
$$
\bigcap_{\delta>0}\overline{\mathrm{co}}\,\tilde{\Gamma}(t,B^\delta_\mathit{w}(r,y))=\tilde{\Gamma}(t,r,y). 
$$ 
\end{assmp}

Assumption \ref{ces} corresponds to the condition imposed in \cite{ba93,ca90} for finite-dimensional control systems. It is satisfied whenever $\tilde{\Gamma}(t,\cdot,\cdot):\R\times X\twoheadrightarrow \R\times \R\times E$ is an upper semicontinuous multifunction for the weak topology of $X$ and the norm topology of $E$ with norm closed, convex values for every $t\in \Omega$; see \cite[Proposition 4.2]{ly95}. A verifiable sufficient condition for the Cesari property is as follows. 

\begin{assmp*}
\label{assmp*}
\begin{enumerate}[\rm (i)]
\item $\Gamma(t,\cdot):X\twoheadrightarrow E$ is an upper semicontinuous multifunction for the weak topology of $X$ and the norm topology of $E$ with norm closed, convex values for every $t\in \Omega$. 
\item $L(t,\cdot,\cdot)F(t,\cdot)$ is lower semicontinuous on $X\times E\times \R$ for the weak topology of $X$ and the norm topology of $E$ for every $t\in \Omega$.
\item $L(t,y,\cdot)F(t,r)$ is convex on $E$ for every $(t,r,y)\in \Omega\times \R\times X$.
\item $f(t,\cdot,\cdot)$ is lower semicontinuous on $X\times E$ for the weak topology of $X$ and the norm topology of $E$ for every $t\in \Omega$.
\item $f(t,y,\cdot)$ is convex on $E$ for every $(t,y)\in \Omega\times E$.  
\end{enumerate}
\end{assmp*}

\begin{thm}
If Assumptions \ref{assmp1} and 3.4$'$hold, then Assumption \ref{ces} does. 
\end{thm}

\begin{proof}
We first show that $\tilde{\Gamma}(t,r,y)$ is norm closed and convex for every $(t,r,y)\in \Omega\times \R\times X$, that is, $\overline{\mathrm{co}}\,\tilde{\Gamma}(t,r,y)=\tilde{\Gamma}(t,r,y)$. To this end, let $(a^1_n,a^2_n,z_n)\in \tilde{\Gamma}(t,r,y)$ for each $n\in \N$ and $(a^1_n,a^2_n,z_n)\to (a^1,a^2,z)\in \R\times \R\times E$. It follows from the definition of $\tilde{\Gamma}(t,r,y)$ that $a^1_n\ge L(t,y,z_n)F(t,r)$, $ a^2_n\ge f(t,y,z_n)$, and $z_n\in \Gamma(t,y)$ for each $n\in \N$. Since $L(t,y,\cdot)F(t,r)$ and $f(t,y,\cdot)$ are lower semicontinuous on $E$ for every $(t,r,y)\in \Omega\times \R\times X$, we have $a^1\ge L(t,y,z)F(t,r)$ and $ a^2\ge f(t,y,z)$. Also, the closedness of $\Gamma(t,y)$ yields $z\in \Gamma(t,y)$. Hence, $(a^1,a^2,z)\in \tilde{\Gamma}(t,r,y)$. To show the convexity of $\tilde{\Gamma}(t,r,y)$, take any $(a^1_0,a^2_0,z_0)$ and $(a^1_1,a^2_1,z_1)$ in $\tilde{\Gamma}(t,r,y)$, and $\lambda\in [0,1]$. We then have $a^1_i\ge L(t,y,z_i)F(t,r)$, $ a^2_i\ge f(t,y,z_i)$, and $z_i\in \Gamma(t,y)$ for $i=0,1$. Since $L(t,y,\cdot)F(t,r)$ and $f(t,y,\cdot)$ are convex on $E$ for every $(t,r,y)\in \Omega\times \R\times X$, we have $\lambda a^1_0+(1-\lambda)a^1_1\ge L(t,y,\lambda z_0+(1-\lambda)z_1)F(t,r)$ and $\lambda a^2_0+(1-\lambda)a^2_1\ge f(t,y,\lambda z_0+(1-\lambda)z_1)$. Also, the convexity of $\Gamma(t,y)$ yields $\lambda z_0+(1-\lambda)z_1\in \Gamma(t,y)$. Hence, $\lambda(a^1_0,a^2_0,z_0)+(1-\lambda)(a^1_1,a^2_1,z_1)\in \tilde{\Gamma}(t,r,y)$. 

To demonstrate the stated equality, let $(t,r,y)\in \Omega\times \R\times X$ be arbitrarily fixed and take any $(a^1,a^2,z)\in \bigcap_{\delta>0}\overline{\mathrm{co}}\,\tilde{\Gamma}(t,B^\delta_\mathit{w}(r,y))$. By choosing $\delta=1/n$ with $n\in \N$, one can extract sequences $\{(r_n,y_n) \}_{n\in \N}$ in $\R\times X$ with $(r_n,y_n)\to (r,y)$ and $\{(a^1_n,a^2_n,z_n) \}_{n\in \N}$ in $\R\times \R\times E$ with $(a^1_n,a^2_n,z_n)\in \overline{\mathrm{co}}\,\tilde{\Gamma}(t,r_n,y_n)=\tilde{\Gamma}(t,r_n,y_n)$ for each $n\in \N$ and $(a^1_n,a^2_n,z_n)\to (a^1,a^2,z)$. We then have $a^1_n\ge L(t,y_n,z_n)F(t,r_n)$, $ a^2_n\ge f(t,y_n,z_n)$, and $z_n\in \Gamma(t,y_n)$ for each $n\in \N$. Since $L(t,\cdot,\cdot)F(t,\cdot)$ and $f(t,\cdot,\cdot)$ are lower semicontinuous on $X\times E\times \R$ for every $t\in \Omega$, we have $a^1\ge L(t,y,z)F(t,r)$ and $ a^2\ge f(t,y,z)$. In view of the weak compactness of $X$, the graph closedness of the compact valued, upper semicontinuous multifunction $\Gamma(t,\cdot):X\twoheadrightarrow E$ yields $z\in \Gamma(t,y)$. Hence, $(a^1,a^2,z)\in \tilde{\Gamma}(t,r,y)$. This implies the inclusion $\bigcap_{\delta>0}\overline{\mathrm{co}}\,\tilde{\Gamma}(t,B^\delta_\mathit{w}(r,y))\subset\tilde{\Gamma}(t,r,y)$. The converse inclusion is obvious. 
\end{proof}

\begin{rem}
To guarantee the integrability of the recursive integrand on $\A$ over the infinite horizon, we impose integrable boundedness on $\Gamma$ in Assumption \ref{assmp2}(i) and the standard linear growth conditions on the integrands $LF$ and $f$ in Assumptions \ref{assmp2}(iii) and (iv), which leads to conditions \eqref{eq1} and \eqref{eq2}. An alternative hypothesis adopted in \cite{bbs89} is that (a) there exists $\gamma\in L^1_{\mathrm{loc}}(\Omega)$ such that $\| x'(t) \|\le \gamma(t)$ for every $x\in \A$ a.e.\ $t\in \Omega$ together with the additional assumptions that (b) $L\ge 0$ and $F\ge 0$; (c) there exists $\psi\in L^1_{\mathrm{loc}}(\Omega)$ with $\psi\ge 0$ such that $f(t,y,z)\ge -\psi(t)$ for every $(t,y,z)\in \mathrm{gph}\,\Gamma$. Conditions (b) and (c) are imposed also in \cite{ba93,ca90} with an alternative growth condition on $LF$. For the relevancy of condition \eqref{eq1} to remove condition (b), see \cite[Remark 6.2]{ca90}. 
\end{rem}

\subsection{An Existence Result}
We are now ready to state the first main result of the paper, which is an infinite-dimensional analogue of \cite{ba93,bbs89,ca90}, whose proof exploits Cesari's lower closure theorem. The direct proof presented here is instead based on the finite horizon truncation and the standard diagonalization procedure for a minimizing sequence. 

\begin{prop}
\label{prop}
Let $E$ be a separable Banach space. Under Assumptions \ref{assmp1}--\ref{ces}, there exists an optimal arc for \eqref{P}. 
\end{prop}

\begin{proof}
Let $\{ x_n \}_{n\in \N}\subset \A$ be a minimizing sequence for \eqref{P}, whose existence is guaranteed in Assumption \ref{assmp2}. Since $G$ is integrably bounded, by Assumption \ref{assmp1}(ii), there exists $\gamma\in L^1(\Omega)$ such that $\| x_n'(t) \|\le \gamma(t)$ for every $t\in \Omega$ and $n\in \N$. In view of $x_n(t)=\int_0^tx_n'(s)ds+\xi_0$, we have $\| x_n(\tau)-x_n(t) \|\le \int_t^\tau\| x_n'(s) \|ds\le \int_t^\tau\gamma(s)ds$ for every $t,\tau\in \Omega$ with $t\le \tau$. Choose any $T>0$ with $t\le T$. Since the finite measure on $[0,T]$ defined by $A\mapsto \int_A\gamma(s)ds$ for $A\subset [0,T]$ is absolutely continuous with respect to the Lebesgue measure, for every $\varepsilon>0$ there exists $\delta>0$ such that $\int_A\gamma(s)ds<\varepsilon$ for every $A\subset [0,T]$ with $|A|<\delta$. This means that $\| x_n(\tau)-x_n(t) \|<\varepsilon$ for each $n\in \N$ whenever $|t-\tau|<\delta$. Therefore, $\{ x_n \}_{n\in \N}$ is equicontinuous with respect to the norm topology of $X$, and hence, so is with respect to the weak topology of $X$. By the Arzela--Ascoli theorem (see \cite[Theorems XII.6.4 and XII.7.2]{du66}), there exist a subsequence of $\{ x_n \}_{n\in \N}$ (which we do not relabel) and a continuous function $x\in C_w(\Omega,X)$ such that $\sup_{t\in [0,T]}\rho(x_n(t),x(t))\to 0$ for every $T>0$. Moreover, for every $T>0$, the sequence $\{ x_n'|_{[0,T]} \}_{n\in \N}$ is bounded, uniformly integrable subset of $L^1([0,T],E)$ such that $x_n'(t)\in G(t)$ a.e.\ $t\in \Omega$ for each $n\in \N$ by Assumption \ref{assmp1}(ii). 

Let $\{ T_n \}_{n\in \N}$ be a monotone increasing sequence of positive real numbers with $T_n\to \infty$. In view of Theorem \ref{diest}, for $T_1$, there exist a subsequence $\{ x_{n_1(k)}' \}_{k\in \N}$ of $\{ x_n' \}_{n\in \N}$ and $u_1\in L^1([0,T_1],E)$ such that $x_{n_1(k)}'|_{[0,T_1]}\to u_1$ weakly in $L^1([0,T_1],E)$. Again by Theorem \ref{diest}, for $T_2$, there exist a subsequence $\{ x_{n_2(k)}' \}_{k\in \N}$ of $\{ x_{n_1(k)}' \}_{k\in \N}$ and $u_2\in L^1([0,T_2],E)$ such that $x_{n_2(k)}'|_{[0,T_2]}\to u_2$ weakly in $L^1([0,T_2],E)$. Note that $u_2(t)=u_1(t)$ for every $t\in [0,T_1]$. Inductively, for each $T_N$ with $N\in \N$ there exist a subsequence $\{ x_{n_N(k)}' \}_{k\in \N}$ of $\{ x_{n_{N-1}(k)}' \}_{k\in \N}$ and $u_N\in L^1([0,T_N],E)$ such that $x_{n_N(k)}'|_{[0,T_N]}\to u_N$ weakly in $L^1([0,T_N],E)$ and $u_N(t)=u_{N-1}(t)$ for every $t\in [0,T_{N-1}]$. For every $t\in \Omega$, let $u(t)=u_N(t)$ whenever $t\le T_N$. By construction, the function $u\in L^1(\Omega,E)$ is well-defined because of $\| u(t) \|\le \gamma(t)$. Let $n(k)=n_k(k)$ for each $k\in \N$. This diagonalization procedure demonstrates that $\{ x_{n(k)} \}_{k\in \N}$ is a subsequence of $\{ x_n \}_{n\in \N}$ such that for every $T>0$: $x_{n(k)}\to x$ uniformly in $C_w([0,T],X)$ and $x_{n(k)}'\to u$ weakly in $L^1([0,T],E)$. We then have
\begin{align*}
\langle x^*,x(t) \rangle=\lim_{k\to \infty}\langle x^*,x_{n(k)}(t)\rangle
& =\lim_{k\to \infty}\int_0^t\langle x^*,x_{n(k)}'(s) \rangle ds+\langle x^*,\xi_0 \rangle \\
& =\int_0^t\langle x^*,u(s) \rangle ds+\langle x^*,\xi_0 \rangle \\
& =\left\langle x^*,\int_0^tu(s)ds+\xi_0 \right\rangle
\end{align*}
for every $x^*\in E^*$ and $t\in [0,T]$, where the second line employs the Lebesgue dominated convergence theorem. This means that $x(t)=\int_0^tu(s)ds+\xi_0$ for every $t\in \Omega$ with $x'=u$. Hence, $x\in W^{1,1}_\mathrm{loc}(\Omega,E)$.  

Define $\varphi_k:\Omega\to \R$ and $\psi_k:\Omega\to \R$ by 
$$
\varphi_k(t)=L(t,x_{n(k)}(t),x_{n(k)}'(t))F\left( t,\int_0^tf(s,x_{n(k)}(s),x_{n(k)}'(s))ds \right)
$$
and 
$$
\psi_k(t)=f(t,x_{n(k)}(t),x_{n(k)}'(t))
$$ 
respectively. Since $x_{n(k)}\in \A$, we have $\varphi_k\in L^1(\Omega)$ by \eqref{eq1} and $\psi_k\in L^1_{\mathrm{loc}}(\Omega)$ by \eqref{eq2}. Since $\{ \varphi_k \}_{k\in \N}$ and $\{ \psi_k \}_{k\in \N}$ are locally integrably bounded sequences in $ L^1_{\mathrm{loc}}(\Omega)$, as in the above argument for the sequence $\{ x_n' \}_{n\in \N}$ in $L^1_{\mathrm{loc}}(\Omega,E)$, there exist $\varphi$ and $\psi$ in $L^1_{\mathrm{loc}}(\Omega)$ such that for some subsequences of $\{ \varphi_k \}_{k\in \N}$ and $\{ \psi_k \}_{k\in \N}$ (which we do not relabel), we have $\varphi_k\to \varphi$ and $\psi_k\to \psi$ weakly in $L^1([0,T])$ for every $T>0$. In particular, $\int_0^t\psi_kds\to \int_0^t\psi ds$ for every $t\in \Omega$. Since $\varphi_k(t)=L(t,x_{n(k)}(t),x_{n(k)}'(t))F(t,\int_0^t\psi_kds)$, for every $\delta>0$ we have
$$
(\varphi_k(t),\psi_k(t),x_{n(k)}'(t))\in \tilde{\Gamma}\left(t,B^\delta_\mathit{w}\left( \int_0^t\psi(s)ds,x(t) \right) \right) \quad \text{a.e.\ $t\in \Omega$}
$$
for sufficiently large $k$. By Lemma \ref{lem1}, for $T_1$ there exist $\varphi^1_k\in L^1_{\mathrm{loc}}(\Omega)$, $\psi^1_k\in L^1_{\mathrm{loc}}(\Omega)$, and $u^1_k\in L^1_{\mathrm{loc}}(\Omega,E)$ for every $k\in \N$ such that $(\varphi^1_k,\psi^1_k,u^1_k)$ is a convex combination of $\{ (\varphi_j,\psi_j,x_{n(j)}')\mid j\ge k \}$ with $\varphi^1_{k_1(j)}(t)\to \varphi(t)$ and $\psi^1_{k_1(j)}(t)\to \psi(t)$ a.e.\ $t\in [0,T_1]$, and $u^1_{k_1(j)}(t)\to x'(t)$ strongly in $E$ a.e.\ $t\in [0,T_1]$ for some subsequence $\{ (\varphi^1_{k_1(j)},\psi^1_{k_1(j)},u^1_{k_1(j)}) \}_{j\in \N}$ of $\{ (\varphi^1_k,\psi^1_k,u^1_k) \}_{k\in \N}$. For $T_2$, there exist $\varphi^2_k\in L^1_{\mathrm{loc}}(\Omega)$, $\psi^2_k\in L^1_{\mathrm{loc}}(\Omega)$, and $u^2_k\in L^1_{\mathrm{loc}}(\Omega,E)$ for every $k\in \N$ such that $(\varphi^2_k,\psi^2_k,u^2_k)$ is a convex combination $\{ (\varphi_j,\psi_j,x_{n(j)}')\mid j\ge k \}$ with $\varphi^2_{k_2(j)}(t)\to \varphi(t)$ and $\psi^2_{k_2(j)}(t)\to \psi(t)$ a.e.\ $t\in [0,T_2]$, and $u^2_{k_2(j)}(t)\to x'(t)$ strongly in $E$ a.e.\ $t\in [0,T_2]$ for some subsequence $\{ (\varphi^2_{k_2(j)},\psi^2_{k_2(j)},u^2_{k_2(j)}) \}_{j\in \N}$ of $\{ (\varphi^2_{k_1(j)},\psi^2_{k_1(j)},u^2_{k_1(j)}) \}_{j\in \N}$. Inductively, for each $T_N$ with $N\in \N$ there exists $\varphi^N_k\in L^1_{\mathrm{loc}}(\Omega)$, $\psi^N_k\in L^1_{\mathrm{loc}}(\Omega)$, and $u^N_k\in L^1_{\mathrm{loc}}(\Omega,E)$ for every $k\in \N$ such that $(\varphi^N_k,\psi^N_k,u^N_k)$ is a convex combination of $\{ (\varphi_j,\psi_j,x_{n(j)}')\mid j\ge k \}$ with $\varphi^N_{k_N(j)}(t)\to \varphi(t)$ and $\psi^N_{k_N(j)}(t)\to \psi(t)$ a.e.\ $t\in [0,T_N]$, and $u^N_{k_N(j)}(t)\to x'(t)$ strongly in $E$ a.e.\ $t\in [0,T_N]$ for some subsequence $\{ (\varphi^N_{k_N(j)},\psi^N_{k_N(j)},u^N_{k_N(j)}) \}_{j\in \N}$ of $\{ (\varphi^N_{k_{N-1}(j)},\psi^N_{k_{N-1}(j)},u^N_{k_{N-1}(j)}) \}_{j\in \N}$. Let $(\hat{\varphi}_i,\hat{\psi}_i,\hat{u}_i)=(\varphi^i_{k_i(i)},\psi^i_{k_i(i)},u^i_{k_i(i)})$ for each $i\in \N$. Then $(\hat{\varphi}_i,\hat{\psi}_i,\hat{u}_i)$ is a convex combination of $\{ (\varphi_j,\psi_j,x_{n(j)}')\mid j\ge i \}$ such that $\hat{\varphi}_i(t)\to \varphi(t)$ and $\hat{\psi}_i(t)\to \psi(t)$ a.e.\ $t\in \Omega$, and $\hat{u}_i(t)\to x'(t)$ strongly in $E$ a.e.\ $t\in \Omega$. Hence, for every $\delta>0$ we have
$$
(\hat{\varphi}_i(t),\hat{\psi}_i(t),\hat{u}_i(t))\in \overline{\mathrm{co}}\,\tilde{\Gamma}\left(t,B^\delta_\mathit{w}\left( \int_0^t\psi (s)ds,x(t) \right) \right) \quad \text{a.e.\ $t\in \Omega$}
$$
for sufficiently large $i$. It follows from the Cesari property that 
\begin{align*}
(\varphi(t),\psi(t),x'(t))
& \in \bigcap_{\delta>0}\overline{\mathrm{co}}\,\tilde{\Gamma}\left(t,B^\delta_\mathit{w}\left( \int_0^t\psi(s)ds,x(t) \right) \right) \\
& =\tilde{\Gamma}\left(t,\int_0^t\psi(s)ds,x(t) \right) \quad\text{a.e.\ $t\in \Omega$}.
\end{align*}
Thus, $x'(t)\in \Gamma(t,x(t))$ a.e.\ $t\in \Omega$, and hence, $x\in \A$. Furthermore, $\varphi(t)\ge L(t,x(t),x'(t))F(t,\int_0^t\psi(s)ds)$ and $\psi(t)\ge f(t,x(t),x'(t))$ a.e.\ $t\in \Omega$.   

We claim that $x$ is an optimal arc for \eqref{P}. To this end, denote $\hat{\varphi}_i$ as a convex combination by $\hat{\varphi}_i=\sum_{j\ge i}\lambda^i_j\varphi_j$ with $\sum_{j\ge i}\lambda^i_j=1$ and $\lambda^i_j\ge 0$, where for each $i\in \N$ only finitely many $\lambda^i_j$ are nonzero in the sum. Since $|\hat{\varphi}_i(t)|\le \alpha(t)$ for every $t\in \Omega$ and $i\in \N$ by \eqref{eq1}, the Lebesgue dominated convergence theorem and Assumption \ref{assmp2}(ii) yield 
\begin{align*}
\lim_{i\to \infty}\int_\Omega\hat{\varphi}_i(t)dt
& =\int_\Omega\varphi(t)dt \\
& \ge \int_\Omega L(t,x(t),x'(t))F\left(t,\int_0^t\psi(s)ds \right) \\
& \ge \int_\Omega L(t,x(t),x'(t))F\left(t,\int_0^tf(s,x(s),x'(s))ds \right). 
\end{align*}
On the other hand, since $\{ x_{n(i)} \}_{i\in \N}$ is a minimizing subsequence of $\{ x_n \}_{n\in \N}$, we obtain $\int\varphi_jdt\to \min\eqref{P}$, and hence,  
$$
\lim_{i\to \infty}\int_\Omega\hat{\varphi}_i(t)dt=\lim_{i\to \infty}\sum_{j\ge i}\lambda^i_j\int_\Omega\varphi_j(t)dt=\min\eqref{P}. 
$$
Therefore, $x$ is an optimal arc for \eqref{P}. 
\end{proof}

\section{Differential Inclusions with an Infinite Horizon}
\subsection{Viable Solutions with a Finite Horizon}
To legitimate Assumption \ref{assmp3}, we provide in this section a sufficient condition for the existence of a solution in $W^{1,1}_\mathrm{loc}(\Omega,E)$ to the initial value problem governed by the differential inclusion with an infinite horizon
\begin{equation}
\label{IVP}
x'(t)\in \Gamma(t,x(t)) \text{ a.e.\ $t\in \Omega$ and }x(0)=\xi_0\in X \tag{IVP}
\end{equation}
so that it is consistent with the upper semicontinuity and convexity hypotheses on $\Gamma$ in Assumptions 3.4$'$(i). To this end, we investigate the truncated initial value problem 
\begin{equation}
\label{IVPT}
x'(t)\in \Gamma(t,x(t)) \text{ a.e.\ $t\in [0,T]$ and }x(0)=\xi_0\in X \tag{$\mathrm{IVP}_T$}
\end{equation}
with a finite horizon to construct a solution to \eqref{IVP} from that to \eqref{IVPT} following the similar diagonalization procedure to the proof of Proposition \ref{prop}. For the investigation of \eqref{IVPT}, two aspects which are unmentioned in the previous section must be taken into account explicitly. 

The first aspect is the viability problem: We must select an arc $x$ from $W^{1,1}([0,T],E)$ satisfying the viability constraint $x(t)\in X$ for every $t\in [0,T]$, where $X$ is called a \textit{viability set}. Evidently, the smaller is $X$, the more restrictive is this selection procedure. The tangential condition originated in \cite{na42} that was applied to  ordinary differential equations in finite dimensional systems is adapted to the differential inclusion. Let $X$ be a nonempty closed subset of $E$ and $d_X:E\to \R$ be the distance function given by $d_X(y)=\inf_{\xi\in X}\| \xi-y \|$. The (\textit{Bouligand}) \textit{contingent cone} $T_X(\xi)$ to $X$ at $\xi\in X$ is defined by
$$
T_X(\xi)=\left\{ y\in E\mid \liminf_{\lambda\downarrow 0}\frac{d_X(\xi+\lambda y)}{\lambda}=0 \right\}.
$$
The following are basic properties of the contingent cone; see \cite[Proposition 4.1]{de92}.
\begin{enumerate}[(a)]
\item $T_X(\xi)$ is closed, $0\in T_X(\xi)$, and $rT_X(\xi)\subset T_X(\xi)$ for every $\xi\in X$ and $r\ge 0$; $T_X(\xi)=E$ for every $\xi\in \mathrm{int}\,X$.
\item If $X$ is convex, then $T_X(\xi)$ is convex for every $\xi\in X$ with 
\begin{align*}
T_X(\xi)
& =\left\{ y\in E\mid \lim_{\lambda\downarrow 0}\frac{d_X(\xi+\lambda y)}{\lambda}=0 \right\}=\overline{\{ r(y-\xi)\mid y\in X,\,r\ge 0 \}} \\
& =\left\{ y\in E\mid \langle x^*,y \rangle\le 0 \text{ $\forall x^*\in E^*$ satisfying $\langle x^*,\xi \rangle=\sup_{z\in X}\langle x^*,z\rangle $} \right\}
\end{align*}
where $\overline{\{ \cdots \}}$ denotes the norm closure in $E$. 
\end{enumerate}
For an exhaustive treatment of contingent cones, see \cite[Chapter 4]{af90}. 

The significance of contingent cones to the viability problem in ordinary differential equations arises in the following way: Let $g:X\to E$ be a continuous function and suppose that $x:[0,T)\to E$ is a $C^1$-solution to the autonomous initial value problem $x'(t)=g(x(t))$ with $x(0)=\xi_0$. We then have $x(\tau)=\xi_0+\tau g(\xi_0)+o(\tau)$ and $d_X(x(\tau))=d_X(\xi_0+\tau g(\xi_0))+o(\tau)$ with $\tau^{-1}|o(\tau)|\to 0$ as $\tau\downarrow 0$. Therefore, $g(\xi)\in \{ y\in E\mid \lim_{\lambda\downarrow 0}\lambda^{-1}d_X(\xi+\lambda y)=0 \}\subset T_X(\xi)$ for every $\xi\in X$ is a necessary condition to obtain a $C^1$-solution viable in $X$. In particular, Nagumo's theorem states that if $E$ is an $n$-dimensional Euclidean space, then the above initial value problem has a viable solution in $X$ if and only if $g(y)\in T_X(y)$ for every $y\in X$; see \cite[Theorem, p.\,175]{ac84}. For the nonautonomous differential inclusion under investigation, we impose the condition in the sequel the Nagumo type condition: 
$$
\Gamma(t,y)\cap T_X(y)\ne \emptyset \quad\text{for every $(t,y)\in [0,T]\times X$}.
$$

The second aspect is noncompactness of the set value $\Gamma(t,y)$ with respect to the norm topology: While $\Gamma(t,y)$ is assumed to be bounded and closed for every $(t,y)\in [0,T]\times X$, it is not assumed to be norm compact, which is automatic in the finite-dimensional case. To impose a reasonable bound on the gauge of the noncompactness of $\Gamma$, we introduce the measure of noncompactness for bounded sets in $E$. Let $\B$ be the family of bounded subsets of an infinite-dimensional Banach space $E$. If $B\in \B$ is not relatively compact, then there exists an $\varepsilon>0$ such that $B$ cannot be covered by a finitely many sets of diameter less than $\varepsilon$, where $\mathrm{diam}\,S=\sup_{y_1,y_2\in S}\| y_1-y_2 \|$, $S\subset E$. Thus, it is natural to introduce the following notion: Define the nonadditive set function $\nu:\B\to [0,\infty)$ by
$$
\nu(B)=\inf\{\varepsilon>0\mid \text{$B$ admits a finite cover $\{ S_i \}_{i=1}^n$ with $\mathrm{diam}\,S_i\le \varepsilon$} \}
$$
which is called the (\textit{Kuratowski}) \textit{measure of noncompactness}. Since $\nu\equiv 0$ whenever $E$ is finite dimensional, the measure of noncompactness is sensitive only when $E$ is infinite dimensional. Note that $\nu$ satisfies the following properties; see \cite[Proposition 7.2]{de85}. 
\begin{enumerate}[(a)]
\item $\nu(B)=0$ if and only if $B$ is relatively compact in $E$. 
\item $\nu(rB)=|r|\nu(B)$ for every $r\in \R$ and $B\in \B$; $\nu(B_1+B_2)\le \nu(B_1)+\nu(B_2)$ for every $B_1,B_2\in \B$. 
\item $B_1\subset B_2$ implies $\nu(B_1)\le \nu(B_2)$; $\nu(B_1\cup B_2)=\max\{ \nu(B_1),\nu(B_2) \}$ for every $B_1,B_2\in \B$. 
\item $\nu(\mathrm{co}\,B)=\nu(B)$ and $\nu(\overline{B})=\nu(B)$ for every $B\in \B$, where $\mathrm{co}\,B$ is the convex hull of $B$.
\item $\nu$ is continuous with respect to the Hausdorff distance. 
\item $\nu(B^r(x))=2r$ for every $r>0$ and $x\in E$. 
\end{enumerate}
In the sequel, we impose the bound on the measure of noncompactness for the bounded set $\Gamma(t,y)$ in terms of $\nu$. 

\begin{thm}[\citet{de88}]
\label{deiml}
Let $X$ be a closed subset of a separable Banach space $E$ and $\Gamma:[0,T]\times X\twoheadrightarrow E$ be a closed convex-valued, upper semicontinuous multifunction for the norm topology of $X$ and $E$ with the following property. 
\begin{enumerate}[\rm (i)]
\item There exists $k>0$ such that 
$$
\sup_{z\in \Gamma(t,y)}\| z \|\le k(1+\| y \|) \quad\text{for every $(t,y)\in [0,T]\in X$}.
$$
\item $\Gamma(t,y)\cap T_X(y)\ne \emptyset$ for every $(t,y)\in [0,T]\times X$.
\item There exists $\varphi\in L^1([0,T])$ such that 
$$
\lim_{\tau\downarrow 0}\nu(\Gamma(J_{t,\tau}\times B))\le \varphi(t)\nu(B) \quad\text{ for every $B\in \B$ and $t\in (0,T]$},
$$
where $J_{t,\tau}=[t-\tau,t+\tau]\cap[0,T]$ with $\tau>0$. 
\end{enumerate}
Then there exists an admissible arc for \eqref{IVPT}.
\end{thm}

For the existence of solutions to the initial value problem in $W^{1,p}([0,T],E)$ for $1<p<\infty$ without any viability constraint under the reflexivity of $E$, see \cite{ma01}.

\subsection{Existence of Admissible Arcs}
Based on Theorem \ref{deiml}, we proceed to prove the existence of admissible arcs for \eqref{IVP} with several additional conditions consistent with those in Section 3. In particular, we strengthen the convexity and the upper semicontinuity of $\Gamma$. 

\begin{assmp}
\label{assmp4}
\begin{enumerate}[\rm (i)]
\item $X\subset E$ is weakly compact and convex. 
\item $\mathrm{gph}\,\Gamma(t,\cdot)=\{ (y,z)\in X\times E\mid z\in \Gamma(t,y) \}$ is convex for every $t\in \Omega$.
\item There exists a relatively weakly compact-valued multifunction $G:\Omega\twoheadrightarrow E$ such that $\Gamma(t,y)\subset G(t)$ for every $(t,y)\in \Omega\times X$. 
\item $\Gamma:\Omega\times X\twoheadrightarrow E$ is an upper semicontinuous multifunction for the weak topology of $X$ and the norm topology of $E$ with norm closed values. 
\item For every $T>0$ there exists $k>0$ such that 
$$
\sup_{z\in \Gamma(t,y)}\| z \|\le k(1+\| y \|) \quad\text{for every $(t,y)\in [0,T]\in X$}.
$$
\item $\Gamma(t,y)\cap T_X(y)\ne \emptyset$ for every $(t,y)\in \Omega\times X$.
\item There exists $\varphi\in L^1_{\mathrm{loc}}(\Omega)$ such that 
$$
\lim_{\tau\downarrow 0}\nu(\Gamma(J_{t,\tau}\times B))\le \varphi(t)\nu(B) \quad\text{ for every $B\in \B$ and $t\in \Omega$}.
$$
\end{enumerate}
\end{assmp}

The second main result of the paper is as follows, whose proof is more or less parallel to that of Proposition \ref{prop}. 

\begin{prop}
Let $E$ be a separable Banach space. Under Assumption \ref{assmp4}, there exists an admissible arc for \eqref{IVP}.
\end{prop}

\begin{proof}
Since the weakly compact set $X$ is norm bounded, Assumption \ref{assmp4}(v) implies that $\Gamma$ is locally integrably bounded with $\sup_{z\in \Gamma(t,y)}\| z \|\le k(1+\sup_{\xi\in X}\| \xi \|)$ for every $(t,y)\in [0,T]\times X$, where the constant $k$ depends upon $T$. Let $T_n\uparrow \infty$. In view of Theorem \ref{deiml}, for each $n\in \N$ there is a solution $x_n\in W^{1,1}([0,T_n],E)$ to $(\mathrm{IVP}_{T_n})$. Since $\{ x_n|_{[0,T_1]} \}_{n\in \N}$ is an equicontinuous sequence in $C_w([0,T_1],X)$ and $\{ x_n'|_{[0,T_1]} \}_{n\in \N}$ is an integrably bounded sequence in $L^1([0,T_1],E)$ with $x_n'(t)\in G(t)$ for every $t\in [0,T_1]$ and $n\in \N$, by the Arzela--Ascoli theorem and Theorem \ref{diest}, there exist a subsequence $\{ x_{n_1(k)} \}_{k\in \N}$ of $\{ x_n \}_{n\in \N}$, $x_1\in C_w([0,T_1],X)$, and $u_1\in L^1([0,T_1],E)$ such that $x_{n_1(k)}|_{[0,T_1]}\to x_1$ uniformly in $C_w([0,T_1],X)$ and $x_{n_1(k)}'|_{[0,T_1]}\to u_1$ weakly in $L^1([0,T_1],E)$. Since $\{ x_{n_1(k)}|_{[0,T_2]} \}_{k\in \N}$ is an equicontinuous sequence in $C_w([0,T_2],X)$ and $\{ x_{n_2(k)}'|_{[0,T_2]} \}_{k\in \N}$ is an integrably bounded sequence in $L^1([0,T_2],E)$ with $x_n'(t)\in G(t)$ for every $t\in [0,T_2]$ and $n\in \N$, there exist a subsequence $\{ x_{n_2(k)} \}_{k\in \N}$ of $\{ x_{n_1(k)} \}_{k\in \N}$, $x_2\in C_w([0,T_2],X)$, and $u_2\in L^1([0,T_2],E)$ such that $x_{n_2(k)}|_{[0,T_2]}\to x_2$ uniformly in $C_w([0,T_1],X)$ and $x_{n_2(k)}'|_{[0,T_2]}\to u_2$ weakly in $L^1([0,T_2])$. Note that $x_2(t)=x_1(t)$ and $u_2(t)=u_1(t)$ for every $t\in [0,T_1]$. Inductively, for each $T_N$ there exists a subsequence $\{ x_{n_N(k)} \}_{k\in \N}$ of $\{ x_{n_{N-1}(k)} \}_{k\in \N}$, $x_N\in C_w([0,T_N],X)$, and $u_N\in L^1([0,T_N],E)$ such that $x_{n_N(k)}|_{[0,T_N]}\to x_N$ uniformly in $C_w([0,T_N],X)$ and $x_{n_N(k)}'|_{[0,T_N]}\to u_N$ weakly in $L^1([0,T_N],E)$, and $x_N(t)=x_{N-1}(t)$ and $u_N(t)=u_{N-1}(t)$ for every $t\in [0,T_{N-1}]$. For every $t\in \Omega$, let $x(t)=x_N(t)$ and $u(t)=u_N(t)$ whenever $t\le T_N$. By construction, the functions $x\in C_w(\Omega,X)$ and $u\in L^1_{\mathrm{loc}}(\Omega,E)$ are well-defined. Let $n(k)=n_k(k)$ for each $k\in \N$. This diagonalization procedure demonstrates that $\{ x_{n(k)} \}_{k\in \N}$ and $\{ u_{n(k)} \}_{k\in \N}$ are  subsequences of $\{ x_n \}_{n\in \N}$ and $\{ u_n \}_{n\in \N}$ respectively such that for every $T>0$: $x_{n(k)}\to x$ uniformly in $C_w([0,T],X)$ and $u_{n(k)}\to u$ weakly in $L^1([0,T],E)$. As in the same way with the proof of Proposition \ref{prop}, one can show that $x'=u$, and hence, $x\in W^{1,1}_{\mathrm{loc}}(\Omega,E)$. 

By Lemma \ref{lem1}, for every $T>0$ there exists a sequence $\{ \hat{u}_i \}_{i\in \N}$ in $L^1([0,T],E)$ such that $\hat{u}_i$ is a convex combination of $\{ x_{n(k)}'\mid k\ge i \}$ and for some subsequence of $\{ \hat{u}_i \}_{i\in \N}$ (which we do not relabel), we have $\hat{u}_i(t)\to x'(t)$ strongly in $E$ a.e.\ $t\in [0,T]$. Denote $\hat{u}_i$ as a  convex combination by $\hat{u}_i=\sum_{k\ge i}\lambda_k^ix_{n(k)}'$ with $\sum_{k\ge i}\lambda_k^i=1$ and $\lambda_k^i\ge 0$, where for each $i\in \N$ only finitely many $\lambda_k^i$ are nonzero in the sum. Define $\hat{x}_i=\sum_{k\ge i}\lambda_k^ix_{n(k)}$, for which $\hat{x}_i'=\hat{u}_i$. It is evident that $\hat{x}_i(t)\to x(t)$ weakly in $X$ for every $t\in [0,T]$. Since $\hat{u}_i(t)\in \Gamma(t,\hat{x}_i(t))$ a.e.\ $t\in [0,T]$ by the convexity of $\mathrm{gph}\,\Gamma(t,\cdot)$ in Assumption \ref{assmp4}(ii) and $\hat{x}_i(0)=\xi_0$, letting $i\to \infty$ yields at the limit $x'(t)\in \Gamma(t,x(t))$ a.e.\ $t\in [0,T]$ and $x(0)=\xi_0$ by the upper semicontinuity of $\Gamma(t,\cdot)$ in Assumption \ref{assmp4}(iv). Since $T>0$ is arbitrary, we have $x'(t)\in \Gamma(t,x(t))$ a.e.\ $t\in \Omega$. Hence, $x\in \A$. This means that $x$ is a solution to \eqref{IVP}.  
\end{proof}

\end{document}